\documentclass[11pt]{article}

\usepackage{graphicx, amssymb, latexsym, amsfonts, amsmath, lscape, amscd,
	amsthm, color, epsfig, mathrsfs, tikz, enumerate, tabularx}
\usepackage{subcaption}
\usepackage{csquotes}

\usepackage{amssymb}
\usepackage{xcolor}
\usepackage{url}
\usepackage{mathtools}
\usepackage{pgf,tikz}
\usetikzlibrary{decorations.pathreplacing}
\usetikzlibrary{decorations.markings}
\usetikzlibrary{positioning}
\usetikzlibrary{arrows}

\newtheorem{theorem}{Theorem}[section]
\newtheorem{conjecture}[theorem]{Conjecture}
\newtheorem{corollary}[theorem]{Corollary}

\newtheorem{lemma}[theorem]{Lemma}

\newtheorem{question}[theorem]{Question}

\newtheorem{observation}[theorem]{Observation}

\newcommand\DELETE[1]{}

\begin{document}
	
	%\begin{frontmatter}
	
	\title{{\bf On $(n,m)$-chromatic numbers of graphs having bounded sparsity parameters}}
	\author{
		{\sc Sandip Das}$\,^{a}$, {\sc Abhiruk Lahiri}$\,^{b}$,  {\sc Soumen Nandi}$\,^{c}$, \\
		{\sc Sagnik Sen}$\,^{d}$, {\sc S Taruni}$\,^{d}$ \\
		\mbox{}\\
		{\small $(a)$ Indian Statistical Institute Kolkata, India }\\
		{\small $(b)$ Charles University, Czech Republic}\\
		{\small $(c)$ Institute of Engineering \& Management, Kolkata, India}\\
		{\small $(d)$ Indian Institute of Technology Dharwad, India}\\
	}

	\date{}
	
	\maketitle

\begin{abstract}
An $(n,m)$-graph is characterised by having $n$ types of arcs and $m$ types of edges. 
A homomorphism of an $(n,m)$-graph $G$ to an $(n,m)$-graph $H$, is a vertex mapping that preserves adjacency, direction, and type. 
The $(n,m)$-chromatic number of $G$, denoted by $\chi_{n,m}(G)$, is the minimum value of $|V(H)|$ such that there exists a homomorphism of $G$ to $H$. The theory of homomorphisms of $(n,m)$-graphs have connections with graph theoretic concepts like harmonious coloring, nowhere-zero flows; with other mathematical topics like binary predicate logic, Coxeter groups; and has application to the Query Evaluation Problem (QEP) in graph database.

In this article, we show that the arboricity of $G$ is bounded by a function of $\chi_{n,m}(G)$ but not the other way around. 
Additionally, we show that the acyclic chromatic number of $G$ is bounded by a function of $\chi_{n,m}(G)$, a result already known in the reverse direction. 
Furthermore, we prove that the $(n,m)$-chromatic number for the family of graphs with a maximum average degree less than $2+ \frac{2}{4(2n+m)-1}$, including the subfamily of planar graphs with girth at least $8(2n+m)$, equals $2(2n+m)+1$. 
This improves upon previous findings, which proved the $(n,m)$-chromatic number for planar graphs with girth at least $10(2n+m)-4$ is $2(2n+m)+1$.

It is established that the $(n,m)$-chromatic number for the family $\mathcal{T}_2$ of partial $2$-trees is both bounded below and above by quadratic functions of $(2n+m)$, with the lower bound being tight when $(2n+m)=2$. 
We prove $14 \leq \chi_{(0,3)}(\mathcal{T}_2) \leq 15$ and $14 \leq \chi_{(1,1)}(\mathcal{T}_2) \leq 21$ which improves both known lower bounds and the former upper bound. Moreover, for the latter upper bound, to the best of our knowledge we provide the first theoretical proof. 
\end{abstract}


\section{Introduction}
\label{sec intro}
\textbf{Graph homomorphism}~\cite{hell2004graphs} is a fundamental concept in graph theory that captures the structural relationships between two or more graphs. In essence, a graph homomorphism is a vertex mapping from a graph to another graph that preserves the adjacency relationships. 
% Formally, if there exists a homomorphism from graph $G$ to graph $H$, it implies that the vertices of $G$ can be assigned to vertices of $H$ in a way that adjacent vertices in $G$ are mapped to adjacent vertices in $H$.
This concept finds applications in various fields, including complexity theory, artificial intelligence, telecommunication and also in statistical physics and play a crucial role in understanding and solving problems related to graph coloring, graph isomorphism, and other combinatorial optimisation challenges~\cite{hell2004graphs}. The study of graph homomorphisms has led to the development of powerful tools and techniques for analyzing the inherent structure and properties of graphs, making it a valuable area of research within the broader realm of graph theory.

In 2000, Ne\v{s}et\v{r}il and Raspaud~\cite{nevsetvril2000colored} introduced the notion of \textbf{colored homomorphisms of colored mixed graphs}, or simply, 
\textbf{homomorphisms of $(n,m)$-graphs}, as a generalization of homomorphisms of (undirected) graphs, oriented graphs, $2$-edge-colored graphs, and $m$-edge-colored graphs~\cite{alon1998homomorphisms, hell2004graphs, sopena2016homomorphisms}. 
Furthermore, they defined the parameter $(n,m)$-chromatic number, using homomorphism, as a generalization of the ordinary chromatic number. It is observed that the topic has connections with graph theoretic topics like harmonious coloring~\cite{alon1998homomorphisms} and nowhere-zero flows~\cite{BORODIN2004147}, and other mathematical topics like binary predicate logic~\cite{nevsetvril2000colored} and Coxeter groups~\cite{alon1998homomorphisms}. 
	 
Moreover, homomorphisms of $(n,m)$-graphs (and their variants) are natural models for the Query Evaluation Problem (QEP) in graph database~\cite{angles2017foundations, beaudou2019complexity}. It is worth mentioning that graph databases are popularly used in social networks (e.g., Facebook, Twitter), information networks (e.g., World Wide Web, citation of academic papers), technological networks (e.g., internet, Geographic Information Systems or \textit{GIS},
phone networks), and biological networks (e.g., genomics, food web, neural networks)~\cite{angles2008survey}.   

In the existing 
literature, one can notice that the homomorphisms and 
$(n,m)$-chromatic number of sparse graphs have received a 
particular focus. The studies did not always mention the term 
``sparse'' explicitly, rather many times it dealt with the 
relation between $(n,m)$-chromatic number of sparse families 
of graphs (planar, partial $2$-trees, outerplanar), and some 
sparsity  parameters (maximum average degree, treewidth, 
acyclic chromatic number). In this article, we too focus on some sparsity parameters and 
contribute in filling some gaps in the theory of homomorphisms of sparse $(n,m)$-graphs.

\medskip
   
\noindent \textit{Note:} For this article, we will restrict ourselves to studying $(n,m)$-graphs whose underlying graphs are simple, and where $(n,m) \neq (0,1)$, unless otherwise stated. For standard graph theoretic notion, we will follow West~\cite{west2001introduction}. Also, if a notion (connectedness, degree, etc.) that is meaningful for undirected graphs is used for an $(n,m)$-graph, one can assume it to be applicable for the underlying graph instead.

\paragraph{The $(n,m)$-graphs:}
An \textit{$(n,m)$-graph} is a graph $G$ having $n$ different types of arcs and $m$ different types of edges. Given an $(n,m)$-graph $G$, we will denote its set of arcs and edges with 
	$A(G)$ and $E(G)$, respectively, while $und(G)$ will denote its underlying graph. Moreover, each arc of $G$ is labeled using one of 
	the (even) numbers from $\{2,4, \ldots, 2n\}$ and each edge of $G$ is labeled by one of the numbers from 
	$\{2n+1, 2n+2, \ldots, 2n+m\}$. The labels are called the \textit{type} of the arc (resp., edge).
	
If $uv$ is an arc of the type $\alpha$, for some $\alpha \in \{2,4, \ldots, 2n\}$, then $v$ is an \textit{$\alpha$-neighbor} of $u$. 
Equivalently, for practical convenience, 
 we 	view $vu$ as a 
	\textit{reverse arc} labeled $(\alpha-1)$, that is, a reverse arc of the 
	type $(\alpha-1)$, and say that 
	$u$ is an \textit{$(\alpha-1)$-neighbor} of $v$. 
	On the other hand, if  $uv$ is an edge of the type $\alpha$, for some 
	$\alpha \in \{2n+1, 2n+2, \ldots, 2n+m\}$, then $v$ is an \textit{$\alpha$-neighbor} of $u$, and $u$ is a 
\textit{$\alpha$-neighbor} of $v$. 
Thus, notice that, a vertex $u$ 
can have $2n+m$ different types of adjacencies or neighbors, namely, the types $1, 2, 3, \ldots, 2n+m$. 
Furthermore, given any $\alpha \in \{1, 2, \ldots, 2n+m\}$, $N^{\alpha}(u)$ denotes the set of all 
$\alpha$-neighbors of $u$.

\paragraph{Homomorphisms and $(n,m)$-chromatic number:}	A \textit{homomorphism} of an $(n,m)$-graph $G$ to another $(n,m)$-graph $H$ is a function  
	$$f : V(G) \rightarrow V(H)$$ 
	such that if $uv$ is an arc (resp., reverse arc, edge) of $G$, then $f(u)f(v)$ is also an arc (resp., reverse arc,  edge) of $H$ having the same type as $uv$. 
	The notation $G \rightarrow H$ is used to denote that $G$ admits a homomorphism to $H$.

	Using the notion of homomorphism, one can define the chromatic number of (n,m)-graphs that  generalizes the
	chromatic numbers defined for simple graphs, oriented graphs, $m$-edge-colored graphs, etc~\cite{hell2004graphs}. The \textit{$(n,m)$-chromatic number} 
	of an $(n,m)$-graph $G$ is given by
		$$\chi_{n,m}(G) = \min\{|V(H)| : G \rightarrow H\}.$$
	For a simple graph $S$, the $(n,m)$-chromatic number of the graph $S$ is given by
	$$\chi_{n,m}(S) = \max\{\chi_{n,m}(G) : und(G) = S\}.$$
	For a family $\mathcal{F}$ of graphs, the $(n,m)$-chromatic number of $\mathcal{F}$ is given by
	$$\chi_{n,m}(\mathcal{F}) = \max\{\chi_{n,m}(G) : G \in \mathcal{F}\}.$$
	Notice that, the family $\mathcal{F}$ may contain simple or $(n,m)$-graphs. 

 \medskip

As we restrict ourselves to $(n,m)$-graphs having underlying simple graphs, the following notion and its property becomes significant for our study.  
  A \textit{special $2$-path} is $2$-path $uvw$ of $(n,m)$-graph $G$ where 
$v \in N^{\alpha}(u) \cap N^{\beta}(w)$ such that $\alpha \neq \beta$.

\begin{observation}[\cite{bensmail2017analogues}]\label{obs special 2-path}
Two vertices $u$ and $v$ 
cannot have the same image under any homomorphism of $G$ to any $H$, if and only if they are adjacent or connected by a special $2$-path in $G$. 
\end{observation}

\paragraph{Context and motivation:} 
The \textit{arboricity} of a graph $G$, denoted by \textit{arb(G)}, is the minimum $r$ such that the edges of $G$ can be decomposed into $r$ forests.
To the best of our knowledge, the earlier studies 
did not explore the relation between $(n,m)$-chromatic number of graphs and graphs having bounded arboricity, a popular 
sparsity parameter, except for the special case when 
$(n,m)=(1,0)$.

\begin{question}\label{ques arboricity}
    What is the relation between the arboricity of a graph $G$ and its $(n,m)$-chromatic number? 
\end{question}

An  \textit{acyclic $s$-coloring} of a graph $G$ is a proper coloring whose bichromatic subgraphs are forests. The \textit{acyclic chromatic number} $\chi_{a}(G)$ of $G$ is the minimum $s$ such that $G$ admits an acyclic $s$-coloring. 
The parameter acyclic chromatic number is closely related to arboricity, and can be viewed as yet another sparsity parameter. 
A tight upper bound of $s(2n+m)^{s-1}$ is proved~\cite{fabila2008lower, nevsetvril2000colored} for the $(n,m)$-chromatic number of the family of graphs having acyclic chromatic number at most $s$.

\begin{theorem}[\cite{fabila2008lower, nevsetvril2000colored}]\label{th nm-chi vs acyclic}
    Let $\mathcal{A}_s$ denote the family of graphs having acyclic chromatic number at most $s$. Then we have $\chi_{n,m}(\mathcal{A}_s) = s(2n+m)^{s-1}$. 
\end{theorem}

However, whether acyclic chromatic number of a graph is bounded by a function of its $(n,m)$-chromatic number is still open. 

\begin{question}\label{ques acyclic}
    Is acyclic chromatic number of a graph $G$ bounded by its $(n,m)$-chromatic number? 
\end{question}

Borodin~\cite{borodin1979acyclic} showed that the acyclic chromatic number of any planar graph is at most $5$. Thus, using $s=5$ in 
Theorem~\ref{th nm-chi vs acyclic},
it is shown that 
for $\chi_{(n,m)}(\mathcal{P}_3) \leq 5(2n+m)^4$~\cite{nevsetvril2000colored}, 
where $\mathcal{P}_3$ denotes the family of planar graphs. 
The lower bound for this parameter is given by a cubic function of $(2n+m)$, that is, 
$\chi_{(n,m)}(\mathcal{P}_3)=\Omega((2n+m)^3)$.
Thus finding the exact value of $\chi_{(n,m)}(\mathcal{P}_3)$ remains an open question and can be considered an analogue of the Four-Color Conjecture (now a Theorem). In a recent breakthrough, Gu\'{s}piel and Gutowski~\cite{seriesb} have shown that $\chi_{(0,m)}(\mathcal{P}_3) = O(m^3)$.

The best (and the only) known possible analogue of the 
Gr\"{o}tzsch's theorem for $(n,m)$-graphs is a result due to 
Montejano, Pinlou, Raspaud, and Sopena~\cite{montejano2009chromatic}
which shows that the $(n,m)$-chromatic number for the family of planar graphs having girth at least $10(2n+m)-4$ is equal to $2(2n+m)+1$. 

\begin{theorem}[\cite{montejano2009chromatic}]\label{th mixed-planar high girth}
Let $\mathcal{P}_{10(2n+m)-4}$ denote the family of planar graphs having girth at least $10(2n+m)-4$. Then we have $\chi_{n,m}(\mathcal{P}_{10(2n+m)-4})=2(2n+m)+1$. 
\end{theorem}

Borodin, Kim, Kostochka, West~\cite{BORODIN2004147} proved a similar result for
a family of sparse graphs having low maximum average degree and high girth (see~\cite{BORODIN2004147} for the exact result). As a corollary, they
improved Theorem~\ref{th mixed-planar high girth} for $(0,m)$-graphs. 

\begin{corollary}[\cite{BORODIN2004147}]\label{cor borodin sparse}
Let $\mathcal{P}_{\frac{20m-2}{3}}$ denote the family of planar graphs having girth at least $\frac{20m-2}{3}$. Then we have $\chi_{0,m}(\mathcal{P}_{\frac{20m-2}{3}})=2m+1$. 
\end{corollary}

An important conjecture in the theory of nowhere-zero flows is a relaxation of Jagear's Conjecture~\cite{jaeger1984circular} for planar graphs, whose equivalent dual formulation is 
using the notion of circular chromatic number~\cite{vince1988star}. 

\begin{conjecture}\label{conj jagear planar}
Let $G$ be a planar graph with girth $f(g)$.
If $f(g) \geq 4g$, then the circular chromatic number of $G$ is at most 
$2+\frac{1}{g}$.  
\end{conjecture}

Notice that the conjecture is equivalent to the 
Gr\"{o}tzsch's theorem when $g=1$. To date it remains open for all values of $g \geq 2$, even though particular progress for the cases $g =2$ and $3$ are present in the literature~\cite{dvovrak2017density, postle2022density}. 
This conjecture has a rich history of 
general approximation. 
In 1996, Ne\v{s}et\v{r}il and Zhu~\cite{nevsetvril1996bounded}, 
and in 2001,
Galuccio, Goddyn, and Hell~\cite{galluccio2001high} proved Conjecture~\ref{conj jagear planar} for $f(g) \geq 10g - 4$. Later in 2001, Zhu~\cite{zhu2001circular} improved it by proving 
Conjecture~\ref{conj jagear planar} for 
$f(g) \geq 8g-3$. 
Interestingly, in 2004,
Borodin, Kim, Kostochka, and West~\cite{BORODIN2004147}
showed that the proof of Corollary~\ref{cor borodin sparse} implies 
Conjecture~\ref{conj jagear planar} for 
$f(g) \geq \frac{20g-2}{3}$. 
Note that, at the time it was published, this bound was the best known 
approximation of the Jagear's conjecture for planar graphs (Conjecture~\ref{conj jagear planar}). 
In 2013, it was improved by Lov{\'a}sz, Thomassen, Wu, and Zhang~\cite{lovasz2013nowhere} through proving that the circular chromatic number of planar graphs having girth at least $6g$ is at most $2+\frac{1}{g}$.

A direct application of Theorem~\ref{th nm-chi vs acyclic} shows that 
$\chi_{(n,m)}(\mathcal{T}_t) \leq t(2n+m)^{t-1}$, where $\mathcal{T}_t$ denotes the family of graphs having treewidth at most $t$. 
For the specific value $t=1$, the family $\mathcal{T}_t$ is nothing but the family of forests. In that case, the exact values of $\chi_{(n,m)}(\mathcal{T}_1)$ is known~\cite{nevsetvril2000colored}. 
For $\mathcal{T}_2$, that is, the family of \textit{partial $2$-trees}, or \textit{graphs with treewidth bounded by $2$}, or \textit{series-parallel graphs}, or 
\textit{$K_4$-minor-free graphs}, we know that
$\chi_{(n,m)}(\mathcal{T}_t) = O((2n+m)^2)$. 

\begin{theorem}[\cite{fabila2008lower, nevsetvril2000colored}]\label{th partial 2-tree known}
	Let $\mathcal{T}_2$ denote the family of all partial $2$-trees. Then we have 
		\begin{align*}
			(2n+m)^2 + 2(2n+m) +  1 \leq  \chi_{n,m}(\mathcal{T}_2) \leq 3(2n+m)^2, &\text{ for $m>0$ even}\\
			(2n+m)^2 + (2n+m) +  1 \leq  \chi_{n,m}(\mathcal{T}_2) \leq 3(2n+m)^2, &\text{ otherwise}.
		\end{align*}
	\end{theorem}

For the particular cases when $2n+m = 2$, 
it is known that the lower bounds are tight~\cite{montejano2010homomorphisms, sopena2016homomorphisms}. 
The best known~\cite{montejano2010homomorphisms} bounds when $2n+m=3$ are 
$13 \leq \chi_{0,3}(\mathcal{T}_{2}) \leq 27$ and 
$13 \leq \chi_{1,1}(\mathcal{T}_{2}) \leq 21$. 
Thus, studying the value of $\chi_{0,3}(\mathcal{T}_{2})$ and 
$\chi_{1,1}(\mathcal{T}_{2})$ are natural open problems.

	\paragraph{Our contributions and organization:}In the following, we provide a section-wise overview of the article highlighting our contributions. 

 \medskip
 
\noindent \textit{Section~\ref{sec arboricity}:}
 We show that the  $(n,m)$-chromatic number of a graph $G$ is not bounded above by a function of its arboricity,
	 yet its arboricity is bounded by a function of its $(n,m)$-chromatic number as a response to Question~\ref{ques arboricity}. 
  We also answer Question~\ref{ques acyclic} positively by showing that 
   the acyclic chromatic number of a graph is bounded above by
   its $(n,m)$-chromatic number.
	 These are generalizations of results due to 
	 Kostochka, Sopena and Zhu~\cite{kostochka1997acyclic} proved 
	for $(n,m) = (1,0)$.

	\medskip
	 
	 \noindent \textit{Section~\ref{sec sp graphs}:}
	We show that 
	the $(n,m)$-chromatic number for the family of graphs having maximum average degree less than 
	$2+ \frac{2}{4(2n+m)-1}$ is equal to $2(2n+m)+1$. 
 As a corollary, we improve the result of Theorem~\ref{th mixed-planar high girth} by showing that the $(n,m)$-chromatic number of the family of planar graphs having girth at least $8(2n+m)$ 
 is also equal to $2(2n+m)+1$. 
 This, in turn, implies that the  circular chromatic number of planar graphs having girth at least $8g$ is at most $2+\frac{1}{g}$, which weakly supports Conjecture~\ref{conj jagear planar}. 
 Even though better approximations of Conjecture~\ref{conj jagear planar} are known, given the long history of research progress on it, it is worth mentioning this corollary.

\medskip

	 \noindent \textit{Section~\ref{sec partial 2 trees}:}  We study the $(n,m)$-chromatic numbers of  the family $\mathcal{T}_2$ of partial $2$-trees where $2n+m=3$. In particular, we show that  
 $\chi_{0,3}(\mathcal{T}_2) \in [14,15]$ and $\chi_{1,1}(\mathcal{T}_2) \in [14,21]$. In the process of proving the above results, we improve both the previously known lower bounds and the first upper bound. Moreover, we provide the first theoretical proof for the second upper bound.

\medskip
  
	  \noindent\textit{Section~\ref{sec conclusions}:}  We share our concluding remarks and propose some future research directions.

	\bigskip
	
\noindent \textit{Note:} The results of Section~\ref{sec arboricity} was part of 
CALDAM 2017~\cite{DBLP:conf/caldam/DasNS17} and Section~\ref{sec sp graphs} and Section~\ref{sec partial 2 trees} was part of 
EuroComb 2021~\cite{lahiri2021chromatic}. This version is significantly enhanced with rigorous proof details and with a corrected version of Theorem $9$ in EuroComb 2021~\cite{lahiri2021chromatic}, in turn serving as an erratum to the result.

% In the EuroComb 2021 paper~\cite{lahiri2021chromatic},  
% $\chi_{1,1}(\mathcal{T}_2) \leq 16$ was claimed. Unfortunately we found a bug in our proof afterwards and were not able to fix the issue. However, we could manage to prove $\chi_{1,1}(\mathcal{T}_2) \leq 21$ instead, which is the 
% same as reported in~\cite{montejano2010homomorphisms} without proof.  This article will serve as an erratum to that result as well. 

	\section{Arboricity  and acyclic chromatic number}\label{sec arboricity}
In this section, we study the relation among $(n,m)$-chromatic number, arboricity and acyclic chromatic number. First we show that an $(n,m)$-graph of bounded arboricity can have arbitrarily large $(n,m)$-chromatic number. After that we will prove that both arboricity and acyclic chromatic number of a graph $G$ is bounded by a function of $\chi_{n,m}(G)$.

	\begin{theorem}\label{th not bounded by arboricity function}
		For every positive integer $k \geq 2$ and $r \geq 2$, there exists an $(n,m)$-graph $G_k$ having $arb(und(G_k)) \leq r$ and $\chi_{n,m}(G_k) \geq k$. 
	\end{theorem}
	
	\begin{proof}
		Consider the complete graph  $K_k$ on $k$ vertices. 
		For all $(n,m) \neq (0,1)$, 
  it is possible to replace all the edges of $K_k$ by a special $2$-path to obtain an $(n,m)$-graph $G_k'$.
		We know that, the end points of the special $2$-path must have different image under any homomorphism of $G_k'$ by Observation~\ref{obs special 2-path}, and thus $\chi_{(n,m)}(G_k') \geq k$. On the other hand, note that $und(G_k')$ has arboricity $2$. Thus, for $r = 2$ take $G_k = G_k'$.  
		For $r > 2$, simply take the disjoint union of the above $G_k'$ with an $(n,m)$-graph $H$
  satisfying $arb(und(H))=r$.
	\end{proof}
	
	Next we show that it is possible to bound the arboricity of an $(n,m)$-graph by a function of $(n,m)$-chromatic number. 
	
		\begin{theorem}\label{Mixed_chromatic-arboricity}
		Let $G$ be a  graph with $\chi_{n,m}(G) = k$. Then $arb(G) \leq \lceil \log_{2n+m} k + \frac{k}{2} \rceil $. 
	\end{theorem}
	
	\begin{proof}
		We know from 
		Nash-Williams' Theorem~\cite{nash1961edge}  that the arboricity $arb(G)$ of any graph $G$ is equal to the maximum of 
		$\lceil \frac{|E(G')|}{|V(G')| - 1} \rceil$ taken over all subgraphs $G'$ of $G$. 
  It is sufficient to prove that for any subgraph $G'$ of $G$, 
		$\frac{|E(G')|}{|V(G')| - 1} \leq \log_{2n+m} k + \frac{k}{2}$.

		Considering $G'$ as a labeled graph,  there are $(2n+m)^{|E(G')|}$ different $(n,m)$-graphs having underlying graph $G'$. 
  As $\chi_{n,m}(G) = k$, for any $(n,m)$-graph $G''$ with $und(G'') = G'$, there exists a homomorphism of $G''$ to an $(n,m)$-graph $G_k$ having $und(G_k) = K_k$. 
		Observe that, even though it is not necessary for 
		$G_k$ to have the complete graph as its underlying graph, we can always add some extra edges or arcs to make $G_k$ have that property. 
		Note that the number of possible homomorphisms of $G''$ to $G_k$ is at most $k^{|V(G'')|} = k^{|V(G')|}$.
		For each such homomorphism of  $G''$ to $G_k$ there are at most $(2n+m)^{k \choose 2}$ different 
  $(n,m)$-graphs with underlying labeled graph  $G'$ as there are $(2n+m)^{k \choose 2}$ choices of $G_k$. 
		Therefore,
		\begin{equation}\label{Mixed_eqn wolog}
			(2n+m)^{k \choose 2}   k^{|V(G')|} \geq (2n+m)^{|E(G')|}
		\end{equation}
		which implies 
		\begin{equation}\label{Mixed_eqn wlog}
			\log_{2n+m} k \geq \frac{|E(G')|}{|V(G')|} - \frac{{k \choose 2}}{|V(G')|}.
		\end{equation}  
If $|V(G')| \leq k$, then 
  $$\frac{|E(G')|}{|V(G')| - 1} \leq \frac{|V(G')|}{2} \leq \frac{k}{2}.$$ 
  If $|V(G')| > k$, then, as $\chi_{n,m}(G') \leq \chi_{n,m}(G) = k$, we have		
		\begin{equation*}
			\begin{split}
				\log_{2n+m} k & \geq \frac{|E(G')|}{|V(G')|} - \frac{k(k - 1)}{2 |V(G')|} \\
				& \geq \frac{|E(G')|}{|V(G')| - 1} - \frac{|E(G')|}{|V(G')|(|V(G')| - 1)} - \frac{k - 1}{2} \\
				& \geq \frac{|E(G')|}{(|V(G')| -1)} - \frac{1}{2} - \frac{k}{2} + \frac{1}{2} \\
				& \geq \frac{|E(G')|}{(|V(G')| -1)} - \frac{k}{2}.
			\end{split}
		\end{equation*}
		
		Therefore, $\frac{|E(G')|}{(|V(G')| -1)} \leq \log_{2n+m}k +\frac{k}{2}$. 
	\end{proof}

	The above two results address Question~\ref{ques arboricity}. Next we will answer Question~\ref{ques acyclic}, however, before that we will prove another interesting relation that connects arboricity, acyclic chromatic number, and  $(n,m)$-chromatic number of a graph.

	\begin{theorem}\label{Mixed_arboricity.chromatic-acyclic}
		Let $G$ be a graph with $arb(G) = r$ and $\chi_{n,m}(G) = k$. 
  Then $\chi_a(G) \leq k^{\lceil \log_{2n+m} r \rceil +1}$.
	\end{theorem}

	\begin{proof}
		Let $v_1, v_2, \ldots, v_g$ be some ordering of the vertices of $G$. 
		Now consider the $(n,m)$-graph $G_0$ with underlying graph $G$ such that for any $i < j$ we have 
		$v_j \in N^{1}(v_i)$ whenever $v_iv_j$ is 
		an edge of $G$. 
		
		Note that the edges of  $G$ can be covered by $r$ edge disjoint forests $F_1, F_2, ..., F_r$ as $arb(G) = r$. 
		For all $i \in \{1,2,\ldots,r\}$, let $b_i$ be the base $2n+m$ representation of the integer $i$.
		 Note that $b_i$ has at most $b = \lceil \log_{2n+m}r \rceil$ digits.

		Now we will construct a sequence of $(n,m)$-graphs 
		$G_1, G_2, \ldots, G_b$ each having underlying graph $G$. 
		For $l \in \{1,2, \ldots,b\}$ we now describe the construction of the $(n,m)$-graph $G_l$. 
		Consider any edge $v_iv_j$ of $G$ where $i <j$. 
		Then $v_iv_j$ is an edge of the forest $F_{q}$ 
  for some $q \in \{1,2, \ldots,r\}$. 
  We denote the $l^{th}$ digit of $b_{q}$  be $l(q)$.
  Then  $G_l$ is constructed from $G_0$ by changing the adjacency type of the edge 
  $v_iv_j$ in such a way that 
		we have $v_j \in N^{l(q)+1}(v_i)$ in $G_l$.
		
		Recall that $\chi_{n,m}(G) \leq k$ and the underlying graph of $G_l$ is $G$. Thus, for each $l \in \{1, 2, \ldots , b \}$ there exists a graph $H_l$ on $k$ vertices and a homomorphism $f_l : G_l \rightarrow H_l$. 
        We denote the vertices of all the $H_l$'s by $1, 2, \ldots, k$. 
		Now we claim that $f(v) = (f_0(v), f_1(v), \ldots, f_b(v))$ for each $v \in V(G)$ is an acyclic coloring of $G$. Notice that the co-domain of $f$ has cardinality $k^{b +1}$.

%----------------------------------------------------

		For adjacent vertices $u,v$ in $G$, clearly, we have 
		$f(u) \neq f(v)$ as $f_0(u) \neq f_0(v)$ in particular. 
		Let $C$ be a cycle in $G$. We have to show that at least three colors have been used to color this cycle with respect to the coloring given by $f$. Note that in $C$ there must be two incident edges $uv$ and $vw$ such that they belong to different forests, 
		say, $F_q$ and $F_{q'}$, respectively.
		Now suppose that $C$ received two colors with respect to $f$.
		Then we must have $f(u) = f(w) \neq f(v)$. In particular we must have 
		$f_0(u) = f_0(w) \neq f_0(v)$. 
		To have that we must also have $u,w \in N^{\alpha}(v)$ for some $\alpha \in \{1,2,\ldots, 2n+m\}$ in $G_0$ (even though $\alpha$ can only take the value $1$ or $2$ in this case). 
		Let  $b_q$ and $b_{q'}$ differ at their $j^{th}$ digit. Thus, $uvw$ must be a special $2$-path in $G_j$.
Hence, by Observation~\ref{obs special 2-path}, we know that
$f_j(u) \neq f_j(w)$, and thus, $f(u) \neq f(w)$, a contradiction. 
		Thus, at least $3$ colors must be used in coloring the cycle $C$. 
	\end{proof}

	Thus, combining   Theorem~\ref{Mixed_chromatic-arboricity} and~\ref{Mixed_arboricity.chromatic-acyclic},
	$\chi_a(G) \leq k^{\lceil \log_{2n+m} \lceil \log_{2n+m}k + \frac{k}{2} \rceil \rceil +1}$ for $\chi_{n,m}(G) = k$. 
	We prove the
	following bound which is better in all cases except for some small values of $k$.

	\begin{theorem}\label{chromatic-acyclic}
		Let $G$ be an $(n,m)$-graph with  $\chi_{n,m}(G) = k$ and  $k \geq 4$. 
  Then 
		$\chi_a(G) \leq k^2 +k^{2+ \lceil \log_{2n+m} \log_{2n+m} k \rceil}$.
	\end{theorem}

	\begin{proof}
		Let $t$ be the maximum real number such that there exists a subgraph $G'$ of $G$ with $|V(G')| \geq k^2$ and $|E(G')| \geq t|V(G')|$.
		Let $G''$ be the biggest subgraph of $G$ with 
  $|E(G'')| > t|V(G'')|$. 
  Thus, by maximality of $t$, $|V(G'')| < k^2$.
		
		Let $G_0 = G - G''$. Hence $\chi_a(G) \leq \chi_a(G_0) + k^2$.
		By maximality of $G''$, for each subgraph $H$ of $G_0$, we have $|E(H)| \leq t|V(H)|$.

		If $t \leq \frac{|V(H)| - 1}{2}$, then 
  $$|E(H)| \leq (t + \frac{1}{2})(|V(H)| -1).$$ 
  If $t > \frac{|V(H)| - 1}{2}$, then $\frac{|V(H)|}{2} < t + \frac{1}{2}$. 
  Thus we have 
  $$|E(H)| \leq \frac{(|V(H)| -1).|V(H)|}{2} \leq (t + \frac{1}{2})(|V(H)| -1).$$ 
  Therefore, $|E(H)| \leq (t + \frac{1}{2})(|V(H)| -1)$ for each subgraph $H$ of $G_0$.
		
		By Nash-Williams' Theorem~\cite{nash1961edge}, there exists $r = \lceil t + \frac{1}{2} \rceil$ forests $F_1, F_2, \ldots, F_r$ which covers all the edges of $G_0$. We know from Theorem~\ref{Mixed_arboricity.chromatic-acyclic}, $\chi_a(G_0) \leq k^{b +1}$ where 
  $b = \lceil \log_{2n+m}r \rceil$.

		Using  inequality~(\ref{Mixed_eqn wlog}) we get $\log_{2n+m} k \geq t - \frac{1}{2}$. Therefore		
		$$b = \lceil \log_{2n+m}(\lceil t + \frac{1}{2}\rceil)\rceil \leq \lceil \log_{2n+m}(1 + \lceil \log_{2n+m}k \rceil)\rceil \leq 1 + \lceil \log_{2n+m} \log_{2n+m}k \rceil.$$
		
		Hence $\chi_a(G) \leq k^2 +k^{2+ \lceil \log_{2n+m} \log_{2n+m}k \rceil}$. 
	\end{proof}

	\section{Maximum average degree}\label{sec sp graphs}
The main result of this section deals with $(n,m)$-chromatic number for the family of planar graphs with a given girth. It is previously known by a result of Monetjano, Pinlou, Raspaud and Sopena~\cite{montejano2009chromatic} that the $(n,m)$-chromatic number for the family of planar graphs having girth at least $10 (2n+m) - 4$ is equal to a linear function of $(2n+m)$. We improve this result by proving that the same bound is attained even for the family of planar graphs having girth at least $8(2n+m)$. To achieve this, we first present a tight bound for $(n,m)$-chromatic number of graphs having bounded maximum average degree.  
	The 
	\textit{maximum average degree} of a graph $G$, denoted by $mad(G)$ is given by 
	$$mad(G) = \max \left\lbrace \frac{2|E(H)|}{|V(H)|} : \text{ $H$ is a subgraph of $G$} \right\rbrace .$$

	\begin{theorem}\label{th mad}
		Let $G$ be a graph with $mad(G) < 2 + \frac{2}{4(2n+m)-1}$. Then  $\chi_{n,m}(G) = 2(2n+m)+1$.
	\end{theorem}

	We begin the proof of Theorem \ref{th mad} by describing a
	complete $(n,m)$-graph on $(2p+1)$ vertices where $p = 2n+m$. 
	We know that there exists a Hamiltonian  decomposition of $K_{2p+1}$ 
	due to Walecki~\cite{alspach2008wonderful} which we are going to describe in the following. 
	To do so, first we will label the vertices of $K_{2p+1}$ is a certain way. 
	Let one specific vertex of it be labeled by the symbol $\infty$ while the other vertices are labeled by the elements of the cyclic group $\mathbb{Z}/2p\mathbb{Z}$. 
	Let $C_0, C_1, \ldots, C_{p-1}$ 
	be the edge disjoint Hamiltonian cycles of the decomposition where 
	$C_j$ is the cycle 
	$$\infty (2p+j) (1+j) (2p-1+j) \ldots (p-1+j) (2p - (p-1)+j) (p+j) \infty$$

	For each 
	$\alpha \in \{2,4, 6 \ldots 2n\}$  convert the cycles $C_{\alpha - 2}$  
	and $C_{\alpha - 1}$ to directed cycles having arcs of color $\alpha$. 
	For each $\alpha \in \{2n+1, n+2, \ldots, 2n+m\}$, 
	convert the cycle $C_{\alpha-1}$
	into a cycle having all edges of color $\alpha$. 
	Thus what we obtain is a complete $(n,m)$-mixed graph on 
	$2p+1$ vertices. We call this so-obtained complete $(n,m)$-graph as $T$.  We now prove a useful property of $T$.

	\begin{lemma}\label{lem key1}
		For every $S \subsetneq V(T)$ we have $|S| <  |N^{\alpha}(S)|$ for all $\alpha \in \{1, 2, \ldots 2n+m \}$. 
	\end{lemma} 
	
	\begin{proof}
		We divide the proof into three parts depending on the value of $\alpha$.

		\medskip
		
	 \textit{Case~1:} 
			If $\alpha \in \{ 2n+1, 2n+2, \ldots 2n+m \}$, then
			assume that $C_{\alpha-1}$ is the cycle $v_1v_2 \ldots v_{2p+1}v_1$
			and the set $S$ consist of vertices 
			$v_{i_1}, v_{i_2}, \ldots, v_{i_l}$ where 
			$i_1 < i_2 < \ldots < i_l$. 
			Now notice that the set of vertices
			$A = \{v_{i_1+1}, v_{i_2+1}, \ldots, v_{i_l+1}\}$  are distinct and are contained in $N^{\alpha}(S)$. On the other hand, note that the set of vertices 
			$B= \{v_{i_1-1}, v_{i_2-1}, \ldots, v_{i_l-1}\}$  are distinct and are contained in $N^{\alpha}(S)$. As $|A| = |B| = |S|$, we are done unless $A= B$. 
			
			If $A = B$, then note that $v_{t} \in S$ implies $v_{t+2} \in S$ where the $+$ operation on indices of $v$ is taken modulo $2p+1$. Hence, if there exists some index $t$ for which we have $v_t, v_{t+1} \in S$, then $S$ must be the whole vertex set, 
			which is not possible. Thus, we must have $v_{i_{j+1}} = v_{i_j+2}$ for all $j \in \{1, 2, \ldots, l\}$ where
			the $+$ operation on indices of $v$ is taken modulo $2p+1$.
			However as $C_{\alpha}$ is an odd cycle on $2p+1$ vertices, 
			it is impossible to satisfy the above condition. Hence $A \neq B$, and we are done in  this case.

		\medskip	
			
		\textit{Case 2:} If $\alpha \in \{2, 4, \ldots, 2n\}$, then 
			observe that $S$ has exactly $|S|$ many $\alpha$-neighbors in 
			$C_{\alpha-2}$ and exactly $|S|$ many $\alpha$-neighbors in 
			$C_{\alpha-1}$. Furthermore, assume that $A$ and $B$ are the sets of $\alpha$-neighbors of $S$ in $C_{\alpha-2}$ and $C_{\alpha-1}$, respectively. As $|A| = |B| = |S|$ and 
			$A \cup B \subseteq N^{\alpha}(S)$, we are done unless $A= B$. 
			As the 
   pairs of the cycles  $C_{\alpha-2}$ and $C_{\alpha-1}$ can be obtained by translation from $C_0$ and $C_1$, respectively, it is enough to prove our claim for 
			$\alpha = 2$. 

   First we will make some observations assuming $A = B$. 
   Using those observations, we want to show that if $S$ is non-empty, then $S = V(T)$. 
   For convenience, let 
   $L=\{1,2,\ldots,p-2\}$ and 
   $U=\{p+1,p+2,\ldots,2p-1\}$.

   \begin{enumerate}[(i)]
    \item 
       $0 \in S \Leftrightarrow 1 \in A \Leftrightarrow 1 \in B \Leftrightarrow \infty \in S$. 

        \item 
       $\infty \in S \Leftrightarrow 0 \in A \Leftrightarrow 0 \in B \Leftrightarrow 2 \in S$. 
       
       \item 
       $p-1 \in S \Leftrightarrow p+1 \in A \Leftrightarrow p+1 \in B \Leftrightarrow p+2 \in S.$ 
        \item 
       $p \in S \Leftrightarrow \infty \in A  \Leftrightarrow \infty \in B \Leftrightarrow p+1 \in S.$ 
       
       \item For $j \in L$,
       $j \in S \Leftrightarrow 2p-j \in A \Leftrightarrow 2p-j \in B \Leftrightarrow j+2 \in S.$  

        \item For $j \in U$,
       $j \in S \Leftrightarrow 2p+1-j \in A  \Leftrightarrow 2p+1-j \in B \Leftrightarrow j+2 \in S.$      
   \end{enumerate}

   If $S$ contains any odd number from $L$, then it contains all odd numbers from $L$ due to (v). 
   Moreover, it contains the odd number from the set $\{p-1,p\}$ by (v), which, in turn by (iii) and (iv), implies that $S$ contains some even number from $U$. Hence by (vi), $S$ contains all even numbers from $U$. 
    Therefore, in particular, we have $2p-2 \in S$ which implies that $0 \in S$. Thus, by (i) and (ii), we have $\infty, 2 \in S$. 
   Notice that, $2 \in S$ implies that all even numbers of $L$ must belong to $S$ by (v). Thus, the even number from the set $\{p-1,p\}$ belongs to $S$ as well, by (v). This, due to (iii) and (iv), implies that some odd number of $U$ belongs to $S$, which, due to (vi), implies that all odd numbers in $U$ belong to $S$. That means, if $S$ contains any odd number from $L$, then $S = V(T)$.

   If $S$ contains any even number from $L$, then it contains all even numbers from $L$ due to (v), and  the even number from the set $\{p-1,p\}$ due to (iii) and (iv). This, due to (vi), implies that $S$ contains all odd numbers from $U$. In particular, $2p-1 \in S$, which implies $1 \in S$ by (vi). 
   Thus, by the above paragraph we have $S = V(T)$ in this case. 
   That means, if $S$ contains any number from $L$, 
   then $S = V(T)$.

   If $S$ contains any number from $U$, then by (vi) $S$ contains at least one number among $p+1$ and $p+2$. Thus, by (iii) and (iv), $S$ must contain one number among  $p-1$ and $p$. Then due to (v),  $S$ must contain some number from $L$, and therefore,   
   $S = V(T)$.
   Hence, if $S$ contains any number from $L \cup U$,
   then $S = V(T)$. 

   If $S$ contains any number from the set $\{0, \infty, p-1, p\}$, then due to (i)-(iv), $S$ must contain some number from $L \cup U$. 
   Therefore, no matter what, if $S$ is non-empty, then $S = V(T)$.

\medskip

	\textit{Case~3:} If $\alpha \in \{1, 3, \ldots, 2n-1\}$, then the case is similar to \textit{Case~2}.

		\medskip

		Therefore, $|S| <  |N^{\alpha}(S)|$ for all $\alpha \in \{1, 2, \ldots 2n+m\}$ and for every $S \subsetneq V(K_{2p+1})$.
	\end{proof}

	Let $T$ be the complete $(n,m)$-graph on $2p+1$ vertices defined after the statement of Theorem~\ref{th mad}, which satisfies the condition of Lemma~\ref{lem key1}. We want to show that $G \to T$ whenever 
	$mad(G)  < 2 + \frac{2}{4p-1}$. 
	That is, it is enough to prove the following lemma. 
	
	\begin{lemma}\label{lem main}
		If $mad(G)  < 2 + \frac{2}{4p-1}$, then  $G \to T$. 
	\end{lemma}

	We will prove the above lemma by contradiction. Hence we assume a minimal 
	(with respect to number of vertices) $(n,m)$-graph $M$
	having $mad(M)  < 2 + \frac{2}{4p-1}$ which does not admit a homomorphism to $T$. 
	We now give some forbidden configurations for $M$ stated as lemmas.

	\begin{lemma}\label{lem deg one}
		The graph $M$ does not contain a vertex having degree one.
	\end{lemma}

	\begin{proof}
		Suppose $M$ contains a vertex $u$ having degree one. Observe that the graph $M'$, 
		obtained by deleting the vertex $u$ from $M$, admits a homomorphism to $T$ due to the minimality of $M$. It is possible to extend the homomorphism $M' \to T$ to a homomorphism of $M \to T$ as any vertex of $T$ has exactly two $\alpha$-neighbors for all $\alpha \in \{1,2, \ldots 2n+m\}$. 
	\end{proof}
	
	A path with all internal vertices of degree two is called a \textit{chain}, and in particular  a \textit{$k$-chain} is a chain having $k$ internal vertices. 
	The endpoints (assume them to always have degree at least $3$)
	of a ($k$-)chain are called \textit{($k$-)chain adjacent}. 
	
% 	\begin{figure}[ht!]
% 		\begin{center}
% 			\begin{tikzpicture}
% 				\node at (0,-0.4) {$u$};
% 				\node at (1,-0.4) {};
% 				\node at (5,-0.4) {};
% 				\node at (6,-0.4) {$v$};
% 				\draw[fill] (0,0) circle[radius=0.1];
% 				\draw[fill] (1,0) circle[radius=0.1];
% 				\draw[fill] (5,0) circle[radius=0.1];
% 				\draw[fill] (6,0) circle[radius=0.1];
% 				\draw[fill] (-1,1) circle[radius=0.1];
% 				\draw[fill] (-1,-1) circle[radius=0.1];
% 				\draw[fill] (7.4,0) circle[radius=0.1];
% 				\draw[fill] (7,1) circle[radius=0.1];
% 				\draw[fill] (7,-1) circle[radius=0.1];
% 				\draw[black] (0,0)--(1,0);
% 				\draw[black] (0,0) -- (-1,1);
% 				\draw[black] (0,0) -- (-1,-1);
% 				\draw[ultra thick] (1,0) -- (5,0) ;
% 				\draw[black] (5,0) -- (6,0);
% 				\draw[black] (6,0) -- (7.4,0);
% 				\draw[black] (6,0) -- (7,1);
% 				\draw[black] (6,0) -- (7,-1);
% 				\draw [decorate,decoration={brace,amplitude=5pt,mirror,raise=4ex}]
% 				(1,0) -- (5,0) node[midway,yshift=-3em]{ $k$ - degree $2$ vertices};

% 			\end{tikzpicture}
% 		\end{center}
% 		\caption{The vertices $u$ and $v$ are $k$-chain adjacent.}
% 	\end{figure}

	\begin{lemma}\label{lem no big chain}
		The graph $M$ does not contain a $k$-chain with $k \geq 2p-1$.
	\end{lemma}

	\begin{proof}
		Suppose $M$ contains a $k$-chain, where $k \geq 2p-1$, with endpoints $u,v$. Observe that the graph $M'$, 
		obtained by deleting the internal degree two vertices from the above mentioned chain from $M$, admits a homomorphism to $T$ due to the minimality of $M$. It is possible to extend the homomorphism $M' \to T$ to a homomorphism of $M \to T$ by Lemma~\ref{lem key1}. 
	\end{proof}

	Let us describe another configuration. 
	Suppose $v$ is chain adjacent to exactly 
	$l$ vertices $v_1, v_2, \ldots, v_l$, each having degree at least three. 
	Let 
	the chains between $v$ and $v_i$ has $k_i$ internal vertices. 
Let us refer to such a configuration as  configuration $C_l$ for convenience.

% 	\begin{figure}[ht!]
% 		\centering
% 		\begin{tikzpicture}[scale=1.0]
% 			\node at (0.4,0.2) {$v$};
% 			\node at (0.3,2.1) {$v_1$};
% 			\node at (1.9,1.5) {$v_2$};
% 			\node at (2.4,-0.1) {$v_3$};
% 			\node at (0.66,-2.4) {};
% 			\node at (1.6,-4.2) {};
% 			\node at (-1.6,-4.2) {};
			
% 			\node at (-1.9,1.3) {$v_l$};
% 			\node at (-5.1,4.4) {};
% 			\node at (-5.6,1.5) {};
% 			\node at (-5.5,-2.9) {};
			
% 			\draw[fill] (0,0) circle[radius=0.1]; %v
% 			\draw[fill] (0,2) circle[radius=0.1]; %v_1
			
% 			\draw[fill] (1.5,1.5) circle[radius=0.1]; %v_2
% 			\draw[fill] (2,0) circle[radius=0.1]; %v_3
% 			\draw[fill] (1.5,-1.5) circle[radius=0.1]; %v_4
% 			\draw[fill] (0,-2) circle[radius=0.1]; %v_5
% 			\draw[fill] (-1.5,-1.5) circle[radius=0.1]; %v_6
% 			\draw[fill] (-2,0) circle[radius=0.1]; %v_7
% 			\draw[fill] (-1.5,1.5) circle[radius=0.1]; %v_l
% 			\draw[fill] (2,0) circle[radius=0.1]; %u_l

% 			\draw[ultra thick] (0,0) -- (0,2);
% 			\draw[ultra thick] (0,0) -- (1.5,1.5);
% 			\draw[ultra thick] (0,0) -- (2,0);
			
% 			\draw[ultra thick] (0,0) -- (1.5,-1.5);
% 			\draw[ultra thick] (0,0) -- (0,-2);
% 			\draw[ultra thick] (0,0) -- (-1.5,-1.5);
% 			\draw[ultra thick] (0,0) -- (-2,0);
% 			\draw[ultra thick] (0,0) -- (-1.5,1.5);
% 			\draw[dashed] (-2,0)--(-1.5,1.5);
% 			\draw [decorate,decoration={brace,amplitude=5pt,mirror,raise=0.4ex}]
% 			(0,0) -- (2,0) node[midway,xshift=3em,yshift=-1em]{$k_3$ internal vertices};
% 		\end{tikzpicture}
% 		\caption{Forbidden Configuration $C_l$ }
% 	\end{figure}
	
	\begin{lemma}\label{lem Csl}
		The graph $M$ does not contain the configuration $C_{l}$ as a induced subgraph if
		$$\sum_{i=1}^{l} k_i > (2p-1)l - 2p$$ where $p = (2n+m)$. 
	\end{lemma}
	
	\begin{proof}
		Suppose $M$ contains the configuration $C_{l}$. Let $M'$ be the graph obtained by deleting all vertices of the configuration except $v_1, v_2, \ldots, v_l$. 
		Thus there exists a homomorphism  $f: M' \to T$ due to minimality of $M$. We are going to extend $f$ to a homomorphism $f_{ext}: M \to T$, which will lead to a contradiction and complete the proof.

		As $T$ has exactly two $\alpha$-neighbors for all $\alpha \in \{1, 2, \ldots 2n+m\}$, 
		Lemma~\ref{lem key1} implies that it is possible to partially extend $f$ to
		the chain between $v$ and $v_i$ in such different ways that will allow us to 
		choose the value of $f_{ext}(v_i)$ from a set of $k_{i}+2$ vertices of $T$. In other words, the value of $f(v_{i})$ forbids at most 
		$2p-k_{i}-1$ values at $v_i$.

		Thus, considering the effects all the chains incident to $v$, 
		at most 
		$$2lp-\sum_{i=1}^lk_{i}-l$$
		values are forbidden at $v$. 
		
		Notice that if this value is less than or equal to $2p$, then it will be possible to extend $f$ to a homomorphism $f_{ext} : M \to T$. That implies the relation $\sum_{i=1}^lk_{i} \leq (2p-1)l - 2p$. 
	\end{proof}

	Now we are ready to start the discharging procedure. 
	First we define a charge function on the vertices of $M$. 
	$$ch(x) = \text{deg}(x) - \left(2 + \frac{2}{4p-1}\right), \text{ for all } x \in V(M).$$ 
	Observe that, $\sum_{x \in V(M)}ch(x) < 0$ as  
	$mad(M) < 2 + \frac{2}{4p-1}$. 
	Now after the completion of the discharging procedure, all updated charge will become non-negative
	implying a contradiction. 
	The discharging rule is the following:
	
	\bigskip
	
	\noindent $(R1):$ \textit{Every vertex having degree three or more donates 
		$\frac{1}{4p-1}$ to the degree two vertices which are part of its incident chains.}

	\bigskip
	
	Let $ch^*(x)$ be the updated charge. Now we are going to calculate this values of this updated charge for vertices of different degrees in $M$.

	\begin{lemma}
		For any degree two vertex $x \in V(M)$, we have $ch^*(x) = 0$.  
	\end{lemma}
	
	\begin{proof}
		As $M$ does not have any degree one vertex due to Lemma~\ref{lem deg one}, every degree two vertex $x$ must be internal vertex of a chain. Thus, by rule $(R1)$ the vertex $x$ must receive $\frac{1}{4p-1}$ charge from each side of the chain. 
		Hence the updated charge is 
		$$ch^*(x) = ch(x) + \frac{2}{4p-1} = deg(x)-2 - \frac{2}{4p-1} + \frac{2}{4p-1} = 0. $$
		Thus we are done. 
	\end{proof}

	\begin{lemma}
		For any vertex $x$ having degree three or more, we have $ch^*(x) \geq 0$.  
	\end{lemma}
	
	\begin{proof}
		Let $x$ be a degree $d$ vertex of $M$. Thus by Lemma~\ref{lem Csl} 
		\begin{align*}
			ch^*(x) &\geq  ch(x) - \frac{(2p-1)d-2p}{4p-1} = d-2 - \frac{2}{4p-1} -\frac{2pd-d-2p}{4p-1}\\
			&= \frac{4pd -8p -d +2 - 2 - 2pd+d+2p}{4p-1}=\frac{2p(d-3)}{4p-1} \geq 0
		\end{align*}
		for $d \geq 3$. 
	\end{proof}

	\bigskip
	
	\noindent \textit{Proofs of Lemma~\ref{lem main} Theorem~\ref{th mad}.} Observe that, due to the above lemmas, we have $0 > \sum_{x \in V(M)} ch(x) = \sum_{x \in V(M)} ch^*(x) \geq 0$, a contradiction.   
	Thus, the proof of Lemma~\ref{lem main} is completed. Lemma~\ref{lem main} directly implies Theorem~\ref{th mad}. \qed

	\bigskip

	As a corollary to this result, we obtain $\chi_{n,m}(\mathcal{P}_g) = 2 (2n+m) + 1$ for all 
	$g \geq 8(2n+m)$.

	\begin{corollary}\label{cor sparse graph}
		Let $\mathcal{P}_g$ denote the family of planar graphs having girth at least $g$. Then for all $g \geq 8(2n+m)$ we have $ \chi_{n,m}(\mathcal{P}_g) = 2 (2n+m) + 1 $.
	\end{corollary}
	
	\begin{proof}
	    	From the Theorem~\ref{th mad} along with the result by Borodin~\cite{borodin1999maximum} that shows that a planar graph $G$ having girth at least $g$ has $mad(G) < \frac{2g}{g-2}$, the following result is obtained as a corollary of Theorem~\ref{th mad}. 
	\end{proof}

 A corollary of the corollary shows that the circular chromatic number of a planar graph having girth at least $8g$
 is at most $2+ \frac{1}{g}$. 
	
	\begin{corollary}\label{cor circular chi}
		Let $G$ be a planar graph with girth $f(g)$.
If $f(g) \geq 8g$, then the circular chromatic number of $G$ is at most 
$2+\frac{1}{g}$.   
	\end{corollary}
	
	\begin{proof}
	    	We know that~\cite{BORODIN2004147} a 
      graph $G$ admits a 
      circular $(2 + \frac{1}{g})$-coloring if and only if 
      $G$ admits a homomorphism to the cycle $C_{2g+1}$ on $2g+1$ vertices.

      Let $G$ be a planar graph having girth at least 
      $f(g) \geq 8g$. Convert $G$ into a $(0,g)$-graph by replacing its edges with edges of type $1$. According to Corollary~\ref{cor sparse graph}, $G$ admits a homomorphism to $T$, the complete $(0,g)$-graph on $2g+1$ vertices defined after the statement of Theorem~\ref{th mad}. However, as all the edges of $G$ are of type $1$, the homomorphism will still be valid if we delete all types of edges except for the ones labeled with $1$ from $T$. However, in that case $T$ can be replaced by a cycle on $2g+1$ vertices having only type $1$ edges.      
	\end{proof}

	\section{Partial $2$-trees}\label{sec partial 2 trees}
The only solutions of $2n+m=2$ are $(n,m)=(1,0)$ and $(0,2)$. Similarly, the only two solutions of 
$2n+m=3$ are $(n,m)=(1,1)$ and $(0,3)$. Therefore, studying the $(n,m)$-chromatic number for these values of $(n,m)$ can help understand the trend for the general values of $(n,m)$.

Recall that, for $2n+m = 2$ case, it is known that the lower bounds are tight~\cite{montejano2010homomorphisms,sopena1997chromatic}. 
The best known bounds when $2n+m=3$ are 
$13 \leq \chi_{0,3}(\mathcal{T}_{2}) \leq 27$ and 
$13 \leq \chi_{1,1}(\mathcal{T}_{2}) \leq 21$~\cite{montejano2010homomorphisms}. 
So, if the trend of lower bound in Theorem~\ref{th partial 2-tree known} being tight were true, then in particular it would be true for the cases when $2n+m=3$ as well. However, we show the contrary via the following result where we improve the lower and the upper bounds for both instances\footnote{The upper bound $\chi_{1,1}(\mathcal{T}_{2}) \leq 21$ was reported in~\cite{montejano2010homomorphisms} without a proof.}.

	\begin{theorem}\label{th partial 2 tree}
		For the family of $\mathcal{T}_2$ of partial $2$-trees we have, 
		\begin{itemize}
			\item[(i)] $14 \leq \chi_{0,3}(\mathcal{T}_{2}) \leq 15$,
			\item[(ii)] $14 \leq \chi_{1,1}(\mathcal{T}_{2}) \leq 21$.
		\end{itemize}
	\end{theorem}
	
	The proof of Theorem~\ref{th partial 2 tree} is contained in a series of lemmas. 

\medskip

Let $\mathcal{F}$ be a family of graphs. Let $T$ be an $(n,m)$-graph satisfying $G \to T$ for all $G$ 
having $und(G) \in \mathcal{F}$. Such a $T$ is an \textit{$(n,m)$-universal bound} of $\mathcal{F}$. 
Moreover, if $T$ does not contain any $(n,m)$-universal bound as a proper subgraph, then $T$ is a \textit{minimal $(n,m)$-universal bound} of $\mathcal{F}$. A family $\mathcal{F}$ of graphs is \textit{complete} if given any finite collection $G_1, G_2, \ldots, G_k \in \mathcal{F}$, there exists a $G \in \mathcal{F}$ containing all $G_i$'s as vertex disjoint subgraphs.  
	
\begin{lemma}
	Let $\mathcal{F}$ be a complete family of graphs. Then there exists a minimal $(n,m)$-universal bound on $\mathcal{F}$ on $\chi_{n,m}(\mathcal{F})$ vertices.
\end{lemma}

\begin{proof}
		Assume the contrary. That is, for each $(n,m)$-graph $T$ on $\chi_{n,m}(\mathcal{F})$ vertices, there exists an $(n,m)$-graph $G_T$ which does not admit a homomorphism to $T$, where $und(G_T) \in \mathcal{F}$.

		Since $\mathcal{F}$ is a complete family of graphs, there exists an $(n,m)$-graph $G$ with $und(G) \in \mathcal{F}$ that contains every such $G_{T}$ as its subgraph.  As $G \in \mathcal{F}$, there is a homomorphism 
		$f : G \to \hat{T}$  for some $(n,m)$-graph $\hat{T}$ on 
		$\chi_{n,m}(\mathcal{F})$ vertices. 
		Then the restriction 
		$$f|_{V(G_{\hat{T}})} : V(G_{\hat{T}}) \to V(\hat{T})$$ 
  is a homomorphism of $G_{\hat{T}}$ to $\hat{T}$. This is a contradiction. 
\end{proof}

An $(n,m)$-graph $T$ has \textit{property $P_{2,1}$} if for any adjacent pair of vertices $u,v$ of $T$ the set $N^{\alpha}(u) \cap N^{\beta}(v)$ is not empty for all $\alpha, \beta \in \{1,2, \ldots, 2n+m\}$. 
It is known~\cite{montejano2010homomorphisms} that if an $(n,m)$-graph has property $P_{2,1}$, then $T$ is an $(n,m)$-universal bound of $\mathcal{T}_2$. We prove the converse for minimal $(n,m)$-universal bound of $\mathcal{T}_2$.

	\begin{lemma}\label{lem partial2-trees}	
	Let $\mathcal{T}_2$ be the family of partial $2$-trees. If $T$ is a minimal $(n,m)$-universal bound of $\mathcal{T}_2$, then $T$ has property $P_{2,1}$.
	\end{lemma}

	\begin{proof}
As $T$ is a minimal $(n,m)$-universal bound of $\mathcal{T}_2$, there exists an $(n,m)$-graph 
		 $G$ with $und(G) \in \mathcal{T}_2$ such that for any homomorphism 
		 $f : G \to T$ the following conditions are satisfied:
   \begin{itemize}
       \item For every $x \in V(T)$, there exists a $u \in V(G)$ such that $f(u) = x$.
       
       \item For every $xy \in A(T) \cup E(T)$, there exists $uv \in A(G) \cup E(G)$ such that $f(u) = x, f(v) = y$.
   \end{itemize}

		For any two adjacent vertices $u,v$ in $G$ and for any $(\alpha, \beta) \in \{1,2, \ldots, 2n+m\}^2$, add a new vertex $w_{\alpha, \beta}$ adjacent to $u$ and $v$ so 
		that we have $w_{\alpha, \beta} \in N^{\alpha}(u) \cap N^{\beta}(v)$. Let the so-obtained $(n,m)$-graph be $G^*$. 
		Observe that, $und(G^*) \in \mathcal{T}_2$ by construction. 
		Therefore, $G^*$ admits a homomorphism to $T$.

		For any pair $x,y$ of adjacent vertices in $T$ and for any homomorphism $f^* : G^* \to T$, there exists $u,v$ in $G$ satisfying 
		$f(u)=x$ and $f(v)=y$. 
		Note that the newly added common neighbors of $u,v$ are connected by a 
		special $2$-path via either $u$ or $v$. Thus, the images of $u, v$ and their newly added common neighbors must be distinct under $f^*$. As $f^*$ is any homomorphism of $G^*$ to $T$, and as 
		 $x,y$ is any pair of adjacent vertices in $T$, 
		 $T$ is forced to have property $P_{2,1}$. 
	\end{proof}
	
	The above lemma implies a necessary and sufficient  condition useful for computing the $(n,m)$-chromatic number of partial $2$-trees.

	\begin{corollary}\label{cor partial 2-tree iff P21}
	   We have $\chi_{n,m}(\mathcal{T}_2) = t$ if and only if there exists 
	   an $(n,m)$-graph $T$ on $t$ vertices with property $P_{2,1}$. 
	\end{corollary}
	   
In light of the above corollary, if one can show that there does not exist any 
$(n,m)$-graph on $t$ vertices with property $P_{2,1}$, then it will imply that 
$\chi_{n,m}(\mathcal{T}_2) \geq t+1$. We will use this observation to prove our lower bounds.

		\begin{lemma}\label{lem partial2trees2}
		If $T$ is a minimal $(n,m)$-universal bound of $\mathcal{T}_{2}$ on $(2n+m)^2 + (2n+m) + 1$ vertices, then every vertex $v$ in $T$ has exactly $(2n+m)+1$ many $\alpha$-neighbors for all $\alpha \in \{1,2, \ldots, 2n+m\}$.
	\end{lemma}
	
	\begin{proof}
		 As  $T$ has property $P_{2,1}$ due to Lemma \ref{lem partial2-trees},
		 each vertex of $T$ has all $(2n+m)$ types of adjacencies. 
		 Let $v$ be an $\alpha$-neighbor of $u$ in $T$. Notice that, 
		 there is at least one vertex which is $\alpha$-neighbor of $u$ and $\beta$-neighbor of $v$. As $\beta$ varies over the set of all $(2n+m)$ types of adjacencies, 
		 $u$ has at least $(2n+m)$ $\alpha$-neighbors, which are also adjacent to $v$. Thus, counting $v$, 
		 $u$ has at least $(2n+m)+1$ many $\alpha$-neighbors. 
		 
		 On the other hand, as $\alpha$ is any of the $(2n+m)$ many adjacencies and $|N^{\alpha}(u)| \geq 2n + m + 1$, we have $|N(u)| \geq (2n+m)(2n+m+1)= (2n+m)^2 + (2n+m)$. As $T$ has only $(2n+m)^2 + (2n+m) + 1$ vertices, the inequalities are tight. 
	\end{proof}
	
	\begin{lemma}\label{lem 4alpha}
		If $T$ is a minimal $(n,m)$-universal bound of $\mathcal{T}_2$ on $(2n+m)^2 + (2n+m) + 1$ vertices, then it can not have $x,y,z \in N^{\alpha}(u)$ such that $x,z$ are $\gamma$-neighbors of $y$.
	\end{lemma}
	
	\begin{proof}
	Suppose $x,y,z \in N^{\alpha}(u)$ and $x,z$ are $\gamma$-neighbors of $y$. We know from the proof of Lemma~\ref{lem partial2-trees} that 
	there is exactly one vertex in the set $N^{\alpha}(u) \cap N^{\beta}(y)$
	for all $\beta$. However, in this case, 
	$\{x,z\} \subseteq N^{\alpha}(u) \cap N^{\gamma}(y)$, a contradiction. 
	\end{proof}

Now we are ready to prove the first lower bound. 
	
	\begin{lemma}\label{lem lb}
	The $(n,m)$-graph $T$ has at least $14$ vertices if either of the following happens. 
		\begin{enumerate}[(i)]
		    \item $T$ is a minimal $(0,3)$-universal bound of $\mathcal{T}_{2}$.
		    
		    \item $T$ is a minimal $(1,1)$-universal bound of $\mathcal{T}_{2}$.
		\end{enumerate}
	\end{lemma}

	\begin{proof}
		Assume the contrary. That is, let $T$ be a minimal $(0,3)$-universal bound 
		or a minimal $(1,1)$-universal bound of $\mathcal{T}_2$ on $13$ vertices. Note that $und(T)$ is a complete graph 
		due to Lemma~\ref{lem partial2trees2}. 
		Also using Lemma~\ref{lem partial2-trees} 
		and Lemma~\ref{lem partial2trees2}
		we know that $T$ contains a $K_3$ with vertices $u,v,w$ (say) whose all edges are of color $3$. 
		
		Next we will try to count the number of vertices in $T$. For convenience, let us denote the set $N^{\alpha}(u) \cap N^{\beta}(v) \cap N^{\gamma}(w) \setminus \{u,v,w\} = A_{\alpha, \beta, \gamma}$. Also, let us denote the set of all common neighbors of $u,v,w$ by $A$. 
		As $T$ has property $P_{2,1}$, there must exist a $x \in A$ which is a $3$-neighbor of $u$ and a $2$-neighbor of $v$. Notice that, if $x$ is a $2$-neighbor or a $3$-neighbor of $w$, then the configuration described in Lemma~\ref{lem 4alpha} is created. Thus, $x$ must be a $1$-neighbor of $w$. Hence, $|A_{3,2,1}| \geq 1.$ 
		Similarly, we can show that 
		\begin{equation}\label{eq alpha-beta-gamma}
		    |A_{1,2,3}|=|A_{1,3,2}|=|A_{2,1,3}|=|A_{2,3,1}|=|A_{3,1,2}|=|A_{3,2,1}| \geq 1.
		\end{equation}

		Until now, among the vertices we have described, none of them is a $2$-neighbor of both $u$ and $v$. 
		However as $T$ has property $P_{2,1}$, there must be a vertex $y$ of $T$ which is a $2$-neighbor of both $u$ and $v$. 
		Note that, $y$ can not be 
		an $3$-neighbor or a $2$-neighbor of $w$ due to Lemma~\ref{lem 4alpha}. Thus, $y$ must be a $1$-neighbor of $w$, to have, $|A_{2,2,1}| \geq 1.$
		Similarly, we can show that 
		\begin{equation}\label{eq alpha-alpha-beta}
		    |A_{2,2,1}|=|A_{1,1,2}|=|A_{2,1,2}|=|A_{1,2,1}|=|A_{1,2,2}|=|A_{2,1,1}| \geq 1.
		\end{equation}
		
		As the sets of the form $A_{\alpha, \beta, \gamma}$ partitions $A$, we can combine equations~(\ref{eq alpha-beta-gamma}) and~(\ref{eq alpha-alpha-beta}) to conclude that $|A| \geq 12$. This implies that $T$ has at least $15$ vertices including $u,v,w$, a contradiction. 
	\end{proof}

       \begin{lemma}\label{lem target 03}
		There exists a $(0,3)$-graph $T_{0,3}$ on $15$ vertices having property $P_{2,1}$. 	\end{lemma}

	\begin{proof}
	  Let $T_{0,3}$ be an $(0,3)$-graph with set of vertices $\mathbb{Z}/5\mathbb{Z} \times \mathbb{Z}/3\mathbb{Z}$. 
	    Let $(i,j)$ and $(i',j')$ be two vertices of $T_{0,3}$. 
	    The adjacency between the vertices are as per the following rules.

	    \begin{itemize}
	   \item If $i = i'$, then $(i,j)$ and $(i',j')$ are not adjacent. 
	         
	   \item If $(i'-i)$ is a non-zero square in $\mathbb{Z}/5\mathbb{Z}$, 
	   then there is an edge of color $(1+j+j')$ (considered modulo $3$) between $(i,j)$ and $(i',j')$. 
	    
	    \item If $(i'-i)$ is not a square in $\mathbb{Z}/5\mathbb{Z}$, 
	   then there is an edge of color $(2+j+j')$ (considered modulo $3$) between $(i,j)$ and $(i',j')$. 
	   \end{itemize}

	    Notice that it is enough to show that $T_{0,3}$ has property $P_{2,1}$. Let $(i,j)$ and $(i',j')$ be any two 
	    adjacent vertices in $T_{0,3}$. 
	    Without loss of generality we may assume that either $i'=i+1$ or $i'=i+2$. We can further assume that $i=0$ and $i'=1$ or $i' = 2$, 
	    still without losing generality. Also, for convenience, let the set $A_{\alpha, \beta}$ be
	    $ N^{\alpha}((i,j)) \cap N^{\beta}((i',j'))$. Thus our objective is to show that all such subsets, 
	    which are a total of nine in number, are non-empty. 
	   
	    \begin{enumerate}[(i)]
	        \item If $i'=1$, then 
	        $(2,j'') \in A_{2+j+j'', 1+j'+j''}$, 
	        $(3,j'') \in A_{2+j+j'', 2+j'+j''}$, and 
	        $(4,j'') \in A_{1+j+j'', 2+j'+j''}$
	        where $j''$ varies over $\mathbb{Z}/3\mathbb{Z}$. Notice that, as $j''$ varies, we obtain a total of nine non-empty subsets of the type $A_{\alpha,\beta}$, and we are done by observing that these subsets have distinct ordered pairs as indices.

	        \item If $i'=2$, then 
	        $(1,j'') \in A_{1+j+j'', 1+j'+j''}$, 
	        $(3,j'') \in A_{2+j+j'', 1+j'+j''}$, and 
	        $(4,j'') \in A_{1+j+j'', 2+j'+j''}$
	        where $j''$ varies over $\mathbb{Z}/3\mathbb{Z}$. Notice that, as $j''$ varies, we obtain a total of nine non-empty subsets of the type $A_{\alpha,\beta}$, and we are done by observing that these subsets have distinct ordered pairs as indices. 
	    \end{enumerate}
	    
	    Hence $T_{0,3}$ has property $P_{2,1}$. 
	\end{proof}

The existence of a $(1,1)$-graph on 
$21$ vertices having property $P_{2,1}$ 
is remarked in the conclusion of~\cite{fabila2008lower}. 
For the sake of completeness, we include an 
explicit construction of the same.

	\begin{lemma}\label{lem target 11}
		There exists a $(1,1)$-graph $T_{1,1}$ on $21$ vertices having property $P_{2,1}$. 	
		\end{lemma}

	\begin{proof}
	    Let $T_{1,1}$ be an $(1,1)$-graph with set of vertices $\mathbb{Z}/7\mathbb{Z} \times \mathbb{Z}/3\mathbb{Z}$. 
	    Let $(i,j)$ and $(i',j')$ be two vertices of $T_{1.1}$. The adjacency between the vertices are as per the following rules. 
	    \begin{itemize}
	       \item If $i = i'$, then, $(i,j)$ and $(i',j')$ are not adjacent.

	        \item If $j'=j$, and $(i'-i)$ is a  non-zero  square in $\mathbb{Z}/7\mathbb{Z}$, 
	    then there is an arc from $(i,j)$ to $(i',j')$.

	    \item If $j' = j +1~(\bmod~3)$ and $(i'- i)$ is a  non-zero  square
	    in $\mathbb{Z}/7\mathbb{Z}$, 
	     then there is an edge between $(i,j)$ and $(i',j')$.
	     
	     \item If $j' = j +1~(\bmod~3)$ and $(i'- i)$ is not a square
	     in $\mathbb{Z}/7\mathbb{Z}$, 
	     then there is an arc from $(i',j')$ to $(i,j)$.
	   \end{itemize}
	  As exactly one among $(i'-i)$ or $(i-i')$ is a  non-zero square in $\mathbb{Z}/7\mathbb{Z}$, the above indeed describes the whole $(1,1)$-graph.    
	     
	   Notice that it is enough to show that $T_{1,1}$ has property $P_{2,1}$. Let $(i,j)$ and $(i',j')$ be any two 
	   adjacent vertices in $T_{1,1}$. 
	  Also, for convenience, let the set $A_{\alpha, \beta}$ be
	    $ N^{\alpha}((i,j)) \cap N^{\beta}((i',j'))$. Thus our objective is to show that all such subsets, 
	    which are a total of nine in number, are non-empty.

	  Let $i$ and $i'$ be two distinct elements of $\mathbb{Z}/7\mathbb{Z}$ and let  
	 $\phi$ and $\varphi$ are two mappings from  $\mathbb{Z}/7\mathbb{Z}$ to itself given by $\phi(x) = (i'-i)^2(x-i)$ and $\varphi(x) = x+1$, respectively. 
    Observe that both $\phi$ and $\varphi$ are bijections.
    Also $x-y$ is a non-zero square in $\mathbb{Z}/7\mathbb{Z}$ if and only if $\phi(x)-\phi(y)$ (resp., $\varphi(x)-\varphi(y)$) is a non-zero square in $\mathbb{Z}/7\mathbb{Z}$. This can be verified by noting $$\phi(x) - \phi(y) = (i'-i)^2 (x-y) \text{ and } \varphi(x) - \varphi(y) = (x-y).$$ 
    If $(i'-i)$ is a non-zero square, then $\phi(i) = 0$ and $\phi(i') = 1$. If $(i'-i)$ is not a non-zero square, then $\varphi(\phi(i)) = 1$ and $\varphi(\phi(i')) = 0$.
	 This brings us to
	 the following cases. 
	 
	   \begin{enumerate}[(i)]
	        \item If $j'=j$, then  without loss of generality we may assume that 
	        $(i,j)=(0,0)$ and $(i',j')=(1,0)$. 
	        
	        \item If $j'=j+1$ (considered modulo $3$) and $(i'-i)$ is a  non-zero square, 
	        then without loss of generality we may assume that 
	        $(i,j)=(0,0)$ and $(i',j')=(1,1)$.

	        \item If $j'=j+1$ (considered modulo $3$) and $(i'-i)$ is not a non-zero square, 
	        then without loss of generality we may assume that 
	        $(i,j)=(1,0)$ and $(i',j')=(0,1)$. 
	    \end{enumerate}
	 
	 \medskip

	         \begin{center}
	          \begin{tabular}{ |c|c||c|c|c|c|c|c|c|c|c| }
	        \hline
	        $(i,j)$  & $(i'j')$ & $A_{1,1}$ & $A_{1,2}$ & $A_{1,3}$ & $A_{2,1}$ & $A_{2,2}$ & $A_{2,3}$ & $A_{3,1}$ & $A_{3,2}$ & $A_{3,3}$ \\
	        \hline
	        $(0,0)$  & $(1,0)$  & $(6,0)$ & $(5,0)$ & $(5,1)$ & $(4,0)$ & $(2,0)$ & $(4,2)$ & $(4,1)$ & $(5,2)$ & $(2,1)$ \\
	        \hline
	        $(0,0)$  & $(1,1)$  & $(6,1)$ & $(5,0)$ & $(6,0)$ & $(4,2)$ & $(2,0)$ & $(4,0)$ & $(4,1)$ & $(2,1)$ & $(5,2)$ \\
	        \hline
	        $(1,0)$  & $(0,1)$  & $(6,1)$ & $(4,0)$ & $(6,0)$ & $(5,2)$ & $(2,0)$ & $(5,0)$ & $(5,1)$ & $(2,1)$ & $(4,2)$ \\
	        \hline 
	     \end{tabular}
	     \end{center}

\medskip	    
	    
 The previously defined nine subsets of the form $A_{\alpha, \beta}$ are non-empty for each of the above listed cases can be observed from the above table. 	 
	  Hence $T_{1,1}$ has property $P_{2,1}$. 
	\end{proof}

\medskip

\noindent \textit{Proof of Theorem~\ref{th partial 2 tree}.}
Using Corollary~\ref{cor partial 2-tree iff P21}, the lower bound follows from  Lemma~\ref{lem lb} and the upper bounds follow from Lemmas~\ref{lem target 03} and~\ref{lem target 11}. \qed

	\section{Conclusions}\label{sec conclusions}
\noindent (1) More often than not, we observe that the bounds of the $(n,m)$-chromatic number is a function of $2n+m$. We wonder 
in which cases $2n+m$ works as an invariant, and in which cases it doesn't.

\medskip

\noindent (2)  We have explored the relation between the parameter $(n,m)$-chromatic number and some graph sparsity parameters,
namely, arboricity, acyclic chromatic number, maximum average degree, treewidth, etc. It is natural to wonder if there is a relation of $(n,m)$-chromatic number with 
twin-width~\cite{bonnet2021twin} 
and 
flip-width~\cite{torunczyk2023flip}, 
the more recently introduced sparsity parameters.

\medskip  

\noindent (3) Let $\mathcal{P}_{g}$ denote the family of planar graphs having girth at least $g$.  It is an interesting natural problem to find the exact girth $g$ for which we will have 
$\chi_{n,m}(\mathcal{P}_{g})=2(2n+m)+1$, but 
$\chi_{n,m}(\mathcal{P}_{g-1})>2(2n+m)+1$. Solving this problem may help us find a better approximation of the  Jagear's conjecture for planar graphs
(Conjecture~\ref{conj jagear planar}).

\medskip

\noindent (4) For the family $\mathcal{T}_2$ of partial $2$-trees, we suspect that the existing lower bound from Theorem~\ref{th partial 2-tree known} is not tight except when $(n,m) = (0,2)$ and $(1,0)$. We proved it formally for $(n,m) = (0,3)$ and $(1,1)$. Can that proof be generalized for all other values of $(n,m)$? 

\medskip

\noindent (5) For the family $\mathcal{T}_2$ of partial $2$-trees, the value of 
$\chi_{n,m}(\mathcal{T}_2)$ seems to be closer to the existing general lower bound than the upper (that is, the bounds from Theorem~\ref{th partial 2-tree known}). Is it possible to find a general construction of an $(n,m)$-graph having property $P_{2,1}$ which has less number of vertices than the upper bound reported in Theorem~\ref{th partial 2-tree known}?

\medskip

\noindent \textbf{Acknowledgements:} This work is partially supported by IFCAM project ``Applications of graph homomorphisms''
(MA/IFCAM/18/39), SERB-MATRICS ``Oriented chromatic and clique number of planar graphs'' (MTR/2021/000858), and
NBHM project ``Graph theoretic model of Channel Assignment Problem (CAP) in wireless network'' (NBHM/RP-8 (2020)/Fresh).

	%    Bibliography.
	%\bibliographystyle{plain}
	%\bibliography{mybibliography.bib}

	\bibliographystyle{abbrv}
	\bibliography{mybibliography}

\end{document}